\documentclass{amsart}
\usepackage{amsfonts}
\usepackage{hyperref}

\newtheorem{theorem}{Theorem}
\theoremstyle{plain}

\newtheorem{application}{Application}

\newtheorem{corollary}{Corollary}

\newtheorem{lemma}{Lemma}

\newtheorem{remark}{Remark}

\numberwithin{equation}{section}
\input{tcilatex}

\begin{document}
\title[Refinements of Mitrinovi\'{c}'s and Cusa's inequality]{Refinements of
Mitrinovi\'{c}-Cusa inequality}
\author{Zhen-Hang Yang}
\address{System Division, Zhejiang Province Electric Power Test and Research
Institute, Hangzhou, Zhejiang, China, 310014}
\email{yzhkm@163.com}
\date{April 10, 2012}
\thanks{This paper is in final form and no version of it will be submitted
for publication elsewhere.}

\begin{abstract}
The Mitrinovi\'{c}-Cusa inequality states that for $x\in \left( 0,\pi
/2\right) $ 
\begin{equation*}
\left( \cos x\right) ^{1/3}<\frac{\sin x}{x}<\frac{2+\cos x}{3}\text{ }
\end{equation*}%
hold. In this paper, we prove that 
\begin{equation*}
\left( \cos x\right) ^{1/3}<\left( \cos px\right) ^{1/\left( 3p^{2}\right) }<%
\frac{\sin x}{x}<\left( \cos qx\right) ^{1/\left( 3q^{2}\right) }<\frac{%
2+\cos x}{3}
\end{equation*}%
hold for $x\in \left( 0,\pi /2\right) $ if and only if $p\in \lbrack
p_{1},1) $ and $q\in (0,1/\sqrt{5}]$, where $p_{1}=0.45346830977067...$. And
the function $p\mapsto \left( \cos px\right) ^{1/\left( 3p^{2}\right) }$ is
decreasing on $(0,1]$. Our results greatly refine the Mitrinovi\'{c}-Cusa
inequality.
\end{abstract}

\maketitle

\section{Introduction}

In the recent past, the following double inequality%
\begin{equation}
\left( \cos x\right) ^{1/3}<\frac{\sin x}{x}<\frac{2+\cos x}{3}\text{ \ }%
\left( 0<x<\frac{\pi }{2}\right)  \label{M-C}
\end{equation}%
has attracted the attention of many scholars.

The left hand side inequality (\ref{M-C}) was first proved by Mitrinovi\'{c}
in \cite{Mitrinovic.1965} (see also \cite[pages 238-240]%
{Mitrinovic.Springer.1970}), and so we call it as \emph{Mitrinovi\'{c}'s
inequality}. While the right hand side inequality (\ref{M-C}) was found by
the German philosopher and theologian Nicolaus de Cusa (1401-1464) and
proved explicitly by Huygens (1629--1695) when he approximated $\pi $, and
it is now known as \emph{Cusa's inequality} \cite{Sandor.RGMIA.8(3)(2005)}, 
\cite{Zhu.CMA.58(2009)}, \cite{Mortitc.MIA.14(2011)}, \cite%
{Neuman.13(4)(2010)}, \cite{Chen.JIA.2011.136}. Hence (\ref{M-C}) can be
called as Mitrinovi\'{c}-Cusa inequality.

A nice refinement of the Mitrinovi\'{c}-Cusa inequality (\ref{M-C}) appeared
in \cite[3.4.6]{Mitrinovic.Springer.1970}. For convenience, we record it as
follows.

\noindent \textbf{Theorem M}. \emph{For }$x\in \left( 0,\pi /2\right) $\emph{%
, }%
\begin{equation}
\cos px\leq \frac{\sin x}{x}\leq \cos qx  \label{1.1}
\end{equation}%
\emph{with the best possible constants}%
\begin{equation*}
\text{\ }p=\frac{1}{\sqrt{3}}\text{ \ and \ \ }q=\frac{2}{\pi }\arccos \frac{%
2}{\pi }.
\end{equation*}%
\emph{Also, the following inequalities hold:}%
\begin{equation}
\cos x\leq \frac{\cos x}{1-x^{2}/3}\leq \left( \cos x\right) ^{1/3}\leq \cos 
\frac{x}{\sqrt{3}}\leq \frac{\sin x}{x}\leq \cos qx\leq \cos \frac{x}{2}\leq
1.  \label{1.2}
\end{equation}

Recently, Kl\'{e}n et al. \cite[Theorem 2.4]{Klen.JIA.2010} showed that the
function $p\mapsto \left( \cos px\right) ^{1/p}$ is decreasing on $\left(
0,1\right) $ and improved Cusa's inequality (the right hand side inequality
in (\ref{M-C})), which is stated as follows.

\noindent \textbf{Theorem K}. \emph{For }$x\in \left( -\sqrt{27/5},\sqrt{27/5%
}\right) $\emph{\ }%
\begin{equation}
\cos ^{2}\frac{x}{2}\leq \frac{\sin x}{x}\leq \cos ^{3}\frac{x}{3}\leq \frac{%
2+\cos x}{3}.  \label{Klen}
\end{equation}

The following sharp bounds for $\left( \sin x\right) /x$ due to Lv et al. 
\cite{Lv.25(2012)} give another refinement of the Mitrinovi\'{c}'s
inequality.

\noindent \textbf{Theorem L}. \emph{For }$x\in \left( 0,\pi /2\right) $\emph{%
\ inequalities }%
\begin{equation}
\left( \cos \tfrac{x}{2}\right) ^{4/3}<\frac{\sin x}{x}<\left( \cos \tfrac{x%
}{2}\right) ^{\theta }  \label{Lv}
\end{equation}%
\emph{hold, where }$\theta =2\left( \ln \pi -\ln 2\right) /\ln 2=\allowbreak
1.\,\allowbreak 303\,0...$\emph{\ and }$4/3$\emph{\ are the best possible
constants.}

Other results involving Mitrinovi\'{c}'s and Cusa's inequality can be found
in \cite{Iyengar.6(1945)}, \cite{Sandor.RGMIA.8(3)(2005)}, \cite%
{Zhu.CMA.58(2009)}, \cite{Wu.75(3-4)(2009)}, \cite{Qi.JIA.2009}, \cite%
{Neuman.13(4)(2010)}, \cite{Chen.JIA.2011.136}, \cite{Mortitc.MIA.14(2011)}, 
\cite{Neuman.JMI.in print} and related references therein.

This paper is motivated by these studies and it is aimed at giving sharp
bounds $\left( \cos px\right) ^{1/\left( 3p^{2}\right) }$ for $\left( \sin
x\right) /x$ to establish interpolated inequalities of (\ref{M-C}), that is,
for $x\in \left( 0,\pi /2\right) $, determine the best $p,q\in \left(
0,1\right) $ such that 
\begin{equation}
\left( \cos x\right) ^{1/3}<\left( \cos px\right) ^{1/\left( 3p^{2}\right) }<%
\frac{\sin x}{x}<\left( \cos qx\right) ^{1/\left( 3q^{2}\right) }<\frac{%
2+\cos x}{3}  \label{M}
\end{equation}%
hold.

The organization of this paper is as follows. Some useful lemmas are given
in section 2. In section 3, the sharp bounds $\left( \cos px\right)
^{1/\left( 3p^{2}\right) }$ for $\left( \sin x\right) /x$ and its relative
error estimates are established. In the last section, some precise estimates
for certain integrals are presented.

\section{Lemmas}

\begin{lemma}[\protect\cite{Vamanamurthy.183.1994}, \protect\cite%
{Anderson.New York. 1997}]
\label{Lemma A}Let $f,g:\left[ a,b\right] \mapsto \mathbb{R}$ be two
continuous functions which are differentiable on $\left( a,b\right) $.
Further, let $g^{\prime }\neq 0$ on $\left( a,b\right) $. If $f^{\prime
}/g^{\prime }$ is increasing (or decreasing) on $\left( a,b\right) $, then
so are the functions 
\begin{equation*}
x\mapsto \frac{f\left( x\right) -f\left( b\right) }{g\left( x\right)
-g\left( b\right) }\text{ \ \ \ and \ \ \ }x\mapsto \frac{f\left( x\right)
-f\left( a\right) }{g\left( x\right) -g\left( a\right) }.
\end{equation*}
\end{lemma}

\begin{lemma}[\protect\cite{Biernacki.9.1955}]
\label{Lemma B}Let $a_{n}$ and $b_{n}$ $(n=0,1,2,...)$ be real numbers and
let the power series $A\left( t\right) =\sum_{n=1}^{\infty }a_{n}t^{n}$ and $%
B\left( t\right) =\sum_{n=1}^{\infty }b_{n}t^{n}$ be convergent for $|t|<R$.
If $b_{n}>0$ for $n=0,1,2,...$, and $a_{n}/b_{n}$ is strictly increasing (or
decreasing) for $n=0,1,2,...$, then the function $A\left( t\right) /B\left(
t\right) $ is strictly increasing (or decreasing) on $\left( 0,R\right) $.
\end{lemma}

\begin{lemma}[{\protect\cite[pp.227-229]{Handbook.math.1979}}]
\label{Lemma C}We have%
\begin{eqnarray}
\cot x &=&\frac{1}{x}-\sum_{n=1}^{\infty }\frac{2^{2n}}{\left( 2n\right) !}%
|B_{2n}|x^{2n-1}\text{, \ }|x|<\pi ,  \label{2.1} \\
\tan x &=&\sum_{n=1}^{\infty }\frac{2^{2n}-1}{\left( 2n\right) !}%
2^{2n}|B_{2n}|x^{2n-1}\text{, \ }|x|<\pi /2,  \label{2.2} \\
\frac{1}{\sin ^{2}x} &=&\frac{1}{x^{2}}+\sum_{n=1}^{\infty }\frac{\left(
2n-1\right) 2^{2n}}{\left( 2n\right) !}|B_{2n}|x^{2n-2}\text{, \ }|x|<\pi ,
\label{2.2a}
\end{eqnarray}%
where $B_{n}$ is the Bernoulli numbers.
\end{lemma}

\begin{lemma}
\label{Lemma D}Let $F_{p}$ be the function defined $\left( 0,\pi /2\right) $
by 
\begin{equation}
F_{p}\left( x\right) =\frac{\ln \frac{\sin x}{x}}{\ln \cos px}.  \label{2.3}
\end{equation}%
Then $F_{p}$ is strictly increasing on $\left( 0,\pi /2\right) $ if $p\in (0,%
\sqrt{5}/5]$ and decreasing on $\left( 0,\pi /2\right) $ if $p\in \left[
1/2,1\right] $. Moreover, we have%
\begin{equation}
\tfrac{\ln 2-\ln \pi }{\ln \left( \cos \frac{1}{2}\pi p\right) }\ln \cos
px<\ln \frac{\sin x}{x}<\tfrac{1}{3p^{2}}\ln \cos px  \label{Main}
\end{equation}%
if $p\in (0,\sqrt{5}/5]$. The inequalities (\ref{Main}) are reversed if $%
p\in \lbrack 1/2,1]$.
\end{lemma}

\begin{proof}
For $x\in \left( 0,\pi /2\right) $, we define $f\left( x\right) =\ln \frac{%
\sin x}{x}$ and $g\left( x\right) =\ln \cos px$, where $p\in (0,1]$. Note
that $f\left( 0^{+}\right) =g\left( 0^{+}\right) =0$, then $F_{p}\left(
x\right) $ can be written as 
\begin{equation*}
F_{p}\left( x\right) =\frac{f\left( x\right) -f\left( 0^{+}\right) }{g\left(
x\right) -g\left( 0^{+}\right) }.
\end{equation*}%
Differentiation and using (\ref{2.1}) and (\ref{2.2}) yield 
\begin{equation*}
\frac{f^{\prime }\left( x\right) }{g^{\prime }\left( x\right) }=\allowbreak 
\frac{p\left( \frac{1}{x}-\cot x\right) }{\tan px}=\allowbreak \frac{%
\sum_{n=1}^{\infty }\frac{2^{2n}}{\left( 2n\right) !}|B_{2n}|x^{2n-1}}{%
\sum_{n=1}^{\infty }\frac{2^{2n}-1}{\left( 2n\right) !}%
p^{2n-2}2^{2n}|B_{2n}|x^{2n-1}}:=\allowbreak \frac{\sum_{n=1}^{\infty
}a_{n}x^{2n-1}}{\sum_{n=1}^{\infty }b_{n}x^{2n-1}},
\end{equation*}%
where 
\begin{equation*}
a_{n}=\frac{2^{2n}}{\left( 2n\right) !}|B_{2n}|\text{, \ }b_{n}=\frac{%
2^{2n}-1}{\left( 2n\right) !}p^{2n-2}2^{2n}|B_{2n}|\text{.}
\end{equation*}%
Clearly, if the monotonicity of $a_{n}/b_{n}$ is proved, then by Lemma \ref%
{Lemma B} it is deduced the monotonicity of $f^{\prime }/g^{\prime }$, and
then the monotonicity of the function $F_{p}$ easily follows from Lemma \ref%
{Lemma A}. Now we prove the monotonicity of $a_{n}/b_{n}$. Indeed,
elementary computation yields%
\begin{eqnarray*}
\frac{b_{n+1}}{a_{n+1}}-\frac{b_{n}}{a_{n}} &=&\left( 2^{2n+2}-1\right)
p^{2n}-\left( 2^{2n}-1\right) p^{2n-2} \\
&=&\left( 4^{n+1}-1\right) p^{2n-2}\left( p^{2}-\frac{1}{4}+\frac{3}{4\left(
4^{n+1}-1\right) }\right) ,
\end{eqnarray*}%
from which it is easy to obtain that for $n\in \mathbb{N}$ 
\begin{equation*}
\frac{b_{n+1}}{a_{n+1}}-\frac{b_{n}}{a_{n}}\left\{ 
\begin{array}{cc}
\leq 0 & \text{if }p^{2}<\frac{1}{5}, \\ 
> & \text{if }p^{2}\geq \frac{1}{4}.%
\end{array}%
\right.
\end{equation*}%
It is seen that $b_{n}/a_{n}$ is decreasing if $0<p\leq \sqrt{5}/5$ and
increasing if $1/2\leq p\leq 1$, which together with $a_{n}$, $b_{n}>0$ for $%
n\in \mathbb{N}$ leads to $a_{n}/b_{n}$ is strictly increasing if $0<p\leq 
\sqrt{5}/5$ and decreasing if $1/2\leq p\leq 1$.

By the monotonicity of the function $F_{p}$ and notice that 
\begin{equation*}
F_{p}\left( 0^{+}\right) =\tfrac{1}{3p^{2}}\text{ \ \ \ and \ \ \ }%
F_{p}\left( \tfrac{\pi }{2}^{-}\right) =\tfrac{\ln 2-\ln \pi }{\ln \left(
\cos \frac{1}{2}\pi p\right) },
\end{equation*}%
the inequalities (\ref{Main}) follow immediately.
\end{proof}

\begin{remark}
Lemma \ref{Lemma D} contains many useful and interesting inequalities for
trigonometric functions. For example, put $p=1/\sqrt{3}$, $\frac{2}{\pi }%
\arccos \frac{2}{\pi }\in \left[ 1/2,1\right] $ in (\ref{Main}) yield the
second and first inequality of (\ref{1.1}), respectively; put $p=1/2\in %
\left[ 1/2,1\right] $ leads to (\ref{Lv}). Similarly, by virtue of Lemma \ref%
{Lemma D} we will easily prove our most main results in the sequel.
\end{remark}

\begin{lemma}
\label{Lemma U(p)}For $x\in \left( 0,\pi /2\right) $, let the function $U:$ $%
(0,1]\mapsto \left( -\infty ,0\right) $ be defined by 
\begin{equation}
U\left( p\right) =\frac{1}{3p^{2}}\ln \cos px.  \label{U(p)}
\end{equation}%
Then $U$ is decreasing on $(0,1]$ with the limit $U\left( 0^{+}\right)
=-x^{2}/6$.$\allowbreak $
\end{lemma}

\begin{proof}
Differentiation yields%
\begin{eqnarray*}
3p^{3}U^{\prime }\left( p\right) &=&-2\ln \left( \cos px\right) -\frac{%
px\sin px}{\cos px}:=V\left( p\right) , \\
V^{\prime }\left( p\right) &=&\allowbreak \frac{x}{2\cos ^{2}px}\left( \sin
2px-2px\right) <0.
\end{eqnarray*}%
It follows that $V\left( p\right) <V\left( 0\right) =0$, and therefore $%
U^{\prime }\left( p\right) >0$, that is, $U$ is decreasing on $(0,1]$.

Simple computation leads to $U\left( 0^{+}\right) =-x^{2}/6$.

Thus the proof ends.
\end{proof}

\begin{lemma}
\label{Lemma f_(p)}For $p\in \left( 0,1\right] $, let the function $f_{p}$
be defined on $\left( 0,\pi /2\right) $ by 
\begin{equation}
f_{p}\left( x\right) :=\ln \frac{\sin x}{x}-\frac{1}{3p^{2}}\ln \cos px.
\label{f_p}
\end{equation}

(i) If $f_{p}\left( x\right) <0$ holds for all $x\in \left( 0,\pi /2\right) $
then $p\in (0,\sqrt{5}/5]$.

(ii) If $f_{p}\left( x\right) >0$ for all $x\in \left( 0,\pi /2\right) $,
then $p\in \lbrack p_{1},1]$, where $p_{1}=0.45346830977067...$ is the
unique root of equation%
\begin{equation}
f_{p}\left( \tfrac{\pi }{2}\right) \allowbreak =\ln \frac{2}{\pi }-\frac{1}{%
3p^{2}}\ln \cos \frac{p\pi }{2}=0  \label{f_p=0}
\end{equation}%
on $\left( 0,1\right] $.
\end{lemma}

\begin{proof}
At first, We assert that there is a unique $p_{1}\in \left( 0,1\right) $ to
satisfy equation (\ref{f_p=0}) such that $f_{p}\left( \tfrac{\pi }{2}\right)
<0$ for $p\in \left( 0,p_{1}\right) $ and $f_{p}\left( \tfrac{\pi }{2}%
\right) >0$ for $p\in (p_{1},1]$.

In fact, Lemma \ref{Lemma U(p)} indicates that $U$ is decreasing on $\left(
0,1\right) $, and so $p\mapsto f_{p}\left( \tfrac{\pi }{2}\right) $ is
increasing on $\left( 0,1\right) $. Since 
\begin{eqnarray*}
f_{1/3}\left( \tfrac{\pi }{2}\right) &=&\allowbreak \ln \tfrac{2}{\pi }-3\ln 
\tfrac{\sqrt{3}}{2}<0, \\
f_{1/2}\left( \tfrac{\pi }{2}\right) &=&\allowbreak \ln \tfrac{2}{\pi }-%
\tfrac{4}{3}\ln \tfrac{\sqrt{2}}{2}>0,
\end{eqnarray*}%
so the equation (\ref{f_p=0}) has a unique solution $p_{1}$ on $\left(
0,1\right) $ and $p_{1}\in \left( 1/3,1/2\right) $ such that $f_{p}\left( 
\tfrac{\pi }{2}\right) <0$ for $p\in \left( 0,p_{1}\right) $ and $%
f_{p}\left( \tfrac{\pi }{2}\right) >0$ for $p\in (p_{1},1]$. Numerical
calculation yields $p_{1}=0.45346830977067...$.

Now, if inequality $f_{p}\left( x\right) <0$ holds for $x\in \left( 0,\pi
/2\right) $, then we have%
\begin{equation*}
\left\{ 
\begin{array}{l}
\lim_{x\rightarrow 0^{+}}\frac{f_{p}\left( x\right) }{x^{4}}%
=\lim_{x\rightarrow 0^{+}}\frac{\ln \frac{\sin x}{x}-\frac{1}{3p^{2}}\ln
\cos px}{x^{4}}=\allowbreak \frac{1}{36}p^{2}-\frac{1}{180}\leq 0, \\ 
f_{p}\left( \tfrac{\pi }{2}^{-}\right) =\allowbreak \ln \tfrac{2}{\pi }-%
\frac{1}{3p^{2}}\ln \left( \cos \frac{1}{2}\pi p\right) \leq 0.%
\end{array}%
\right.
\end{equation*}%
Solving the inequalities for $p$ yields 
\begin{equation*}
p\in (0,\sqrt{5}/5]\cap (0,p_{1}]=(0,\sqrt{5}/5].
\end{equation*}%
In the same way, if inequality $f_{p}\left( x\right) >0$ holds for all $x\in
\left( 0,\pi /2\right) $, then 
\begin{equation*}
p\in \left[ \sqrt{5}/5,1\right] \cap \lbrack p_{1},1]=[p_{1},1],
\end{equation*}%
which completes the proof.
\end{proof}

\section{Main Results}

Now we state and prove the sharp upper bound $\left( \cos px\right)
^{1/\left( 3p^{2}\right) }$\ for $\left( \sin x\right) /x$.

\begin{theorem}
\label{Theorem Mr1}$\allowbreak $For $p\in (0,1]$, the inequality 
\begin{equation}
\frac{\sin x}{x}<\left( \cos px\right) ^{1/\left( 3p^{2}\right) }
\label{Mr1}
\end{equation}%
holds for all $x\in \left( 0,\pi /2\right) $ if and only if $p\in (0,\sqrt{5}%
/5]$. Moreover, we have 
\begin{equation}
\left( \cos \frac{x}{\sqrt{5}}\right) ^{\alpha }<\frac{\sin x}{x}<\left(
\cos \frac{x}{\sqrt{5}}\right) ^{5/3},  \label{Mr1a}
\end{equation}%
where $\alpha =\left( \ln \frac{2}{\pi }\right) /\ln \left( \cos \frac{\sqrt{%
5}\pi }{10}\right) =\allowbreak 1.\,\allowbreak 671\,4...$and $%
5/3=\allowbreak 1.\,\allowbreak 666\,7...$ are the best possible constants.
\end{theorem}

\begin{proof}
From Lemma \ref{Lemma f_(p)} the necessity follows. The second inequality of
(\ref{Main}) implies that the condition $p\in (0,\sqrt{5}/5]$ is sufficient.

Put $p=\sqrt{5}/5$ in (\ref{Main}) yields (\ref{Mr1}).

Thus the proof is completed.
\end{proof}

From the corollary, in order to prove the last inequality in (\ref{M}), it
suffices to compare $e^{-x^{2}/6}$ with $\left( 2+\cos x\right) /3$. We have

\begin{theorem}
\label{Theorem Mr2}The inequality 
\begin{equation}
e^{-x^{2}/6}<\frac{2+\cos x}{3}  \label{Mr2}
\end{equation}%
holds for $x\in \left( 0,\infty \right) $. Moreover, for $x\in \left(
0,a\right) $ ($a>0$) we have%
\begin{equation}
\frac{2+\cos x}{\left( 2+\cos a\right) e^{a^{2}/6}}<e^{-x^{2}/6}<\frac{%
2+\cos x}{3}.  \label{Mr2re}
\end{equation}
\end{theorem}

\begin{proof}
Considering the function $g$ defined by 
\begin{equation*}
g\left( x\right) =\ln \frac{2+\cos x}{3}+\frac{x^{2}}{6},
\end{equation*}%
and differentiation yields%
\begin{eqnarray}
g^{\prime }\left( x\right) &=&\frac{x}{3}-\frac{\sin x}{\cos x+2},
\label{4.4} \\
g^{\prime \prime }\left( x\right) &=&\frac{1}{3}\frac{\left( \cos x-1\right)
^{2}}{\left( \cos x+2\right) ^{2}}\geq 0,  \notag
\end{eqnarray}%
which implies that for $x\in \left( 0,\infty \right) $, $g^{\prime }\left(
x\right) >g^{\prime }\left( 0^{+}\right) =0$, then, $g^{\prime }\left(
x\right) >0$, that is, $g$ is increasing on $\left( 0,\infty \right) $.
Hence, we have $g\left( x\right) >g\left( 0^{+}\right) =0$ for $x\in \left(
0,\infty \right) $, that is, (\ref{Mr2}) is true.

For $x\in \left( 0,a\right) $ we have 
\begin{equation*}
0=g\left( 0^{+}\right) <g\left( x\right) <g\left( a\right) =\ln \left( \frac{%
2+\cos a}{3}e^{a^{2}/6}\right) ,
\end{equation*}%
which proves (\ref{Mr2re}).
\end{proof}

Next we establish the sharp lower bound for $\left( \sin x\right) /x$.

\begin{theorem}
\label{Theorem Ml}$\allowbreak $Let $p\in (0,1]$. Then the inequality 
\begin{equation}
\frac{\sin x}{x}>\left( \cos px\right) ^{1/\left( 3p^{2}\right) }  \label{Ml}
\end{equation}%
holds for all $x\in \left( 0,\pi /2\right) $ if and only if $p\in \lbrack
p_{1},1]$, where $p_{1}=0.45346830977067...$ is the unique root of equation (%
\ref{f_p=0}) in $p\in \left( 0,1\right] $. Moreover, we have 
\begin{equation}
\left( \cos p_{1}x\right) ^{1/\left( 3p_{1}^{2}\right) }<\frac{\sin x}{x}%
<\beta \left( \cos p_{1}x\right) ^{1/\left( 3p_{1}^{2}\right) },
\label{Mlre}
\end{equation}%
where $1$ and $\beta \approx \allowbreak 1.\,\allowbreak 000\,2$ are the
best possible constants.
\end{theorem}

\begin{proof}
\textbf{Necessity}. Lemma \ref{Lemma f_(p)} implies necessity.

\textbf{Sufficiency}. Due to Lemma \ref{Lemma U(p)}, it suffices to show
that $f_{p_{1}}\left( x\right) >0$ for all $x\in \left( 0,\pi /2\right) $,
where $f_{p}$ is defined by (\ref{f_p}). To this end, we introduce an
auxiliary function $h$ defined on $\left( 0,\pi /2\right) $ by 
\begin{equation}
h\left( x\right) =\frac{f_{p_{1}}^{\prime }\left( x\right) }{x^{3}}=\frac{%
\left( \cot x-\frac{1}{x}\right) +\frac{1}{3p_{1}}\tan p_{1}x}{x^{3}}%
\allowbreak .  \label{h(x)}
\end{equation}%
We will show that $h$ is decreasing on $\left( 0,\pi /2\right) $.

Differentiation and simplifying yield 
\begin{equation*}
x^{4}h^{\prime }\left( x\right) =\tfrac{4}{3}\tfrac{x}{\sin ^{2}2p_{1}x}-%
\tfrac{1}{3}\tfrac{x}{\sin ^{2}p_{1}x}+\frac{1}{x}-\tfrac{x}{\sin ^{2}x}%
-3\left( \cot x-\tfrac{1}{x}\right) -\tfrac{\tan p_{1}x}{p_{1}},
\end{equation*}%
which, utilizing (\ref{2.1}), (\ref{2.2}) and (\ref{2.2a}), can be expanded
in power series as 
\begin{eqnarray*}
x^{4}h^{\prime }\left( x\right) &=&\tfrac{4}{3}\sum_{n=1}^{\infty }\tfrac{%
\left( 2n-1\right) 2^{2n}}{\left( 2n\right) !}|B_{2n}|\left( 2p_{1}\right)
^{2n-2}x^{2n-1}-\tfrac{1}{3}\sum_{n=1}^{\infty }\tfrac{\left( 2n-1\right)
2^{2n}}{\left( 2n\right) !}|B_{2n}|p_{1}^{2n-2}x^{2n-1} \\
&&-\sum_{n=1}^{\infty }\tfrac{\left( 2n-1\right) 2^{2n}}{\left( 2n\right) !}%
|B_{2n}|x^{2n-1}-3\sum_{n=1}^{\infty }\tfrac{2^{2n}}{\left( 2n\right) !}%
|B_{2n}|x^{2n-1}-\sum_{n=1}^{\infty }\tfrac{2^{2n}-1}{\left( 2n\right) !}%
2^{2n}|B_{2n}|p_{1}^{2n-2}x^{2n-1} \\
&:&=\sum_{n=1}^{\infty }\frac{2^{2n}|B_{2n}|}{3\left( 2n\right) !}%
u_{n}x^{2n-1},
\end{eqnarray*}%
where 
\begin{equation*}
u_{n}=\left( 2^{2n}-1\right) \left( 2n-10\right) p_{1}^{2n-2}-3\left(
2n-1\right) .
\end{equation*}%
Clearly, $u_{n}<0$ for $n=1,2,3,4,5$. We now show that $u_{n}<0$ for $n\geq
6 $. For this purpose, it needs to prove that for $n\geq 6$ 
\begin{equation*}
p_{1}<\left( \frac{3\left( 2n-1\right) }{\left( 2^{2n}-1\right) \left(
2n-10\right) }\right) ^{\frac{1}{2n-2}}:=h_{1}\left( n\right) .
\end{equation*}%
Since $\left( 2n-1\right) >\left( 2n-10\right) $, we have 
\begin{equation*}
h_{1}\left( n\right) >\left( \frac{3}{2^{2n}-1}\right) ^{\frac{1}{2n-2}}:=%
\text{\ }k\left( n\right) .
\end{equation*}

Considering the function $k:\left( 1,\infty \right) \mapsto \left( 0,\infty
\right) $ defined by%
\begin{equation}
k\left( x\right) =\left( \tfrac{3}{2^{2x}-1}\right) ^{1/\left( 2x-2\right) },
\label{k(x)}
\end{equation}%
and differentiation leads to%
\begin{eqnarray*}
\frac{2\left( x-1\right) ^{2}}{k\left( x\right) }k^{\prime }\left( x\right)
&=&\ln \left( 2^{2x}-1\right) -\ln 3-2\ln 2\frac{\left( x-1\right) 2^{2x}}{%
2^{2x}-1}:=k_{1}\left( x\right) , \\
k_{1}^{\prime }\left( x\right) &=&\allowbreak \frac{2^{2x+2}\ln ^{2}2}{%
\left( 2^{2x}-1\right) ^{2}}\left( x-1\right) ,
\end{eqnarray*}%
which reveals that $k_{1}$ is increasing on $\left( 1,\infty \right) $, and
so $k_{1}\left( x\right) >k_{1}\left( 1^{+}\right) =0$, then $k^{\prime
}\left( x\right) >0$, that is, $k$ is increasing on $\left( 1,\infty \right) 
$. Therefore for $n\geq 6$%
\begin{equation*}
\allowbreak \allowbreak 0.485\,83\approx 1365^{-1/10}=k\left( 6\right) \leq
k\left( n\right) <k\left( \infty \right) =\frac{1}{2}.
\end{equation*}%
It follows that for $n\geq 6$%
\begin{equation*}
h_{1}\left( n\right) >k\left( n\right) >0.485\,83>p_{1},
\end{equation*}%
which indicates that $u_{n}<0$ for $n\geq 6$. Thus we have $h^{\prime
}\left( x\right) <0$, that is, the auxiliary function $h$ is decreasing on $%
\left( 0,\pi /2\right) $.

On the other hand, it is clear that 
\begin{equation*}
h\left( 0^{+}\right) =\lim_{x\rightarrow 0+}\frac{\left( \cot x-\frac{1}{x}%
\right) +\frac{1}{3p_{1}}\tan p_{1}x}{x^{3}}=\allowbreak \frac{1}{9}\left(
p_{1}^{2}-\frac{1}{5}\right) >0.
\end{equation*}%
And we claim that $h\left( \tfrac{\pi }{2}^{-}\right) <0$. If $h\left( 
\tfrac{\pi }{2}^{-}\right) \geq 0$, then there must be $h\left( x\right) >0$
for all $x\in \left( 0,\pi /2\right) $, which, by (\ref{h(x)}), implies that 
$f_{p_{1}}^{\prime }\left( x\right) >0$, then $f_{p_{1}}$ is increasing on $%
\left( 0,\pi /2\right) $. It yields 
\begin{equation*}
f_{p_{1}}\left( x\right) >f_{p_{1}}\left( 0^{+}\right) =0\text{\ and\ }%
f_{p_{1}}\left( x\right) <f_{p_{1}}\left( \tfrac{\pi }{2}\right)
=\allowbreak \ln \tfrac{2}{\pi }-\tfrac{1}{3p_{1}^{2}}\ln \left( \cos \tfrac{%
1}{2}p_{1}\pi \right) =0,
\end{equation*}%
which is a contradiction. Consequently, $h\left( 0^{+}\right) >0$ and $%
h\left( \tfrac{\pi }{2}^{-}\right) <0$.

Make use of the monotonicity of the auxiliary function $h$ it is showed that
there is a unique $x_{0}\in \left( 0,\pi /2\right) $ to satisfy $h\left(
x_{0}\right) =0$ such that $h\left( x\right) >0$ for $x\in \left(
0,x_{0}\right) $ and $h\left( x\right) <0$ for $x\in \left( x_{0},\pi
/2\right) $. Then, by (\ref{h(x)}), it is seen that $f_{p_{1}}$ is
increasing on $\left( 0,x_{0}\right) $ and decreasing on $\left( x_{0},\pi
/2\right) $. It is concluded that 
\begin{eqnarray*}
0 &=&f_{p_{1}}\left( 0^{+}\right) <f_{p_{1}}\left( x\right) <f_{p_{1}}\left(
x_{0}\right) \text{ for }x\in \left( 0,x_{0}\right) , \\
0 &=&f_{p_{1}}\left( \tfrac{\pi }{2}^{-}\right) <f_{p_{1}}\left( x\right)
<f_{p_{1}}\left( x_{0}\right) \text{ for }x\in \left( x_{0},\pi /2\right) ,
\end{eqnarray*}%
that is, $0<f_{p_{1}}\left( x\right) <f_{p_{1}}\left( x_{0}\right) $ for $%
x\in \left( 0,\pi /2\right) $.

Solving the equation $h\left( x\right) =0$ which is equivalent with 
\begin{equation*}
f_{p_{1}}^{\prime }\left( x\right) =\left( \cot x-\frac{1}{x}\right) +\frac{1%
}{3p_{1}}\tan p_{1}x=0\allowbreak
\end{equation*}%
by using mathematical computer software, we find that 
\begin{equation*}
x_{0}\in \left( 1.31187873615727632,1.31187873615727633\right) ,
\end{equation*}%
and $\beta =\exp \left( f_{p_{1}}\left( x_{0}\right) \right) \approx
\allowbreak 1.\,\allowbreak 000\,2$, which proves the sufficiency and (\ref%
{Mlre}).
\end{proof}

Letting $p=\frac{1}{2},\tfrac{1}{\sqrt{3}},\tfrac{1}{\sqrt{2}},\tfrac{\sqrt{6%
}}{3},1$ in Theorem \ref{Theorem Ml} and $p=\tfrac{1}{\sqrt{6}}$, $\frac{1}{3%
}$, $\tfrac{1}{2\sqrt{3}}$, $\frac{1}{4}$, ..., $\rightarrow 0$ in Theorem %
\ref{Theorem Mr1} together with Theorem \ref{Theorem Mr2} we have

\begin{corollary}
\label{Corollary M}For $x\in \left( 0,\pi /2\right) $, we have 
\begin{eqnarray*}
\left( \cos x\right) ^{1/3} &<&\cdot \cdot \cdot <\left( \cos \tfrac{\sqrt{6}%
x}{3}\right) ^{1/2}<\left( \cos \tfrac{x}{\sqrt{2}}\right) ^{2/3}<\cos 
\tfrac{x}{\sqrt{3}}<\left( \cos \tfrac{x}{2}\right) ^{4/3} \\
&<&\left( \cos p_{1}x\right) ^{1/\left( 3p_{1}^{2}\right) }<\frac{\sin x}{x}%
<\left( \cos \tfrac{x}{\sqrt{5}}\right) ^{5/3}<\left( \cos \tfrac{x}{\sqrt{6}%
}\right) ^{2}<\left( \cos \tfrac{x}{3}\right) ^{3} \\
&<&\left( \cos \tfrac{x}{2\sqrt{3}}\right) ^{4}<\left( \cos \tfrac{x}{4}%
\right) ^{16/3}<\cdot \cdot \cdot <e^{-x^{2}/6}<\frac{2+\cos x}{3}.
\end{eqnarray*}%
where $p_{1}=0.45346830977067...$.
\end{corollary}

Thus it can be seen that our results greatly refine Mitrinovi\'{c}-Cusa
inequality (\ref{M-C}).

The following give a relative error estimating $\left( \sin x\right) /x$ by $%
\left( \cos px\right) ^{1/\left( 3p^{2}\right) }$.

\begin{theorem}
\label{Theorem Mrre}$\allowbreak $For $p\in (0,1]$, let $f_{p}$ be defined
on $\left( 0,\pi /2\right) $ by (\ref{f_p}). Then $f_{p}$ is decreasing if $%
p\in $ $(0,\sqrt{5}/5]$ and increasing if $p\in \left[ 1/2,1\right] $.

Moreover, if $p\in $ $(0,\sqrt{5}/5]$ then for $x\in \left( 0,c\right) $
with $c\in \left( 0,\pi /2\right) $ 
\begin{equation}
\gamma _{p}\left( c\right) \left( \cos px\right) ^{1/\left( 3p^{2}\right) }<%
\frac{\sin x}{x}<\left( \cos px\right) ^{1/\left( 3p^{2}\right) }
\label{Mrre}
\end{equation}%
with the best possible constants $\gamma _{p}\left( c\right) =c^{-1}\left(
\sin c\right) \left( \cos pc\right) ^{-1/\left( 3p^{2}\right) }$ and $1$.
The inequalities (\ref{Mrre}) are reversed if $p\in \left[ 1/2,1\right] $.
\end{theorem}

\begin{proof}
Differentiation and using (\ref{2.1}) and (\ref{2.2}) yield 
\begin{eqnarray*}
f_{p}^{\prime }\left( x\right) &=&\allowbreak \left( \cot x-\frac{1}{x}%
\right) +\frac{1}{3p}\tan px \\
&=&-\sum_{n=1}^{\infty }\frac{2^{2n}}{\left( 2n\right) !}|B_{2n}|x^{2n-1}+%
\frac{1}{3}\sum_{n=1}^{\infty }\frac{2^{2n}-1}{\left( 2n\right) !}%
p^{2n-2}2^{2n}|B_{2n}|x^{2n-1} \\
&=&\sum_{n=1}^{\infty }\frac{\left( 2^{2n}-1\right) 2^{2n}}{3\left(
2n\right) !}|B_{2n}|\left( p^{2n-2}-\tfrac{3}{2^{2n}-1}\right)
x^{2n-1}:=\sum_{n=2}^{\infty }s_{n}t_{n}x^{2n-1},
\end{eqnarray*}%
where 
\begin{eqnarray*}
s_{n} &=&\frac{\left( 2^{2n}-1\right) 2^{2n}|B_{2n}|}{3\left( 2n\right) !}%
\tfrac{p^{2n-2}-\frac{3}{2^{2n}-1}}{p-\left( \frac{3}{2^{2n}-1}\right)
^{1/\left( 2n-2\right) }}>0, \\
t_{n} &=&p-k\left( n\right)
\end{eqnarray*}%
for $n\geq 2$ and $p\in (0,1]$, where the function $k$ is defined by (\ref%
{k(x)}). As showed in the proof of Theorem \ref{Theorem Ml}, $k$ is
increasing on $\left( 1,\infty \right) $, and so for $n\geq 2$ 
\begin{equation*}
1/\sqrt{5}=k\left( 2\right) \leq k\left( n\right) \leq k\left( \infty
\right) =\lim_{n\rightarrow \infty }\left( \tfrac{3}{2^{2n}-1}\right)
^{1/\left( 2n-2\right) }=\allowbreak \frac{1}{2},
\end{equation*}%
and then, $t_{n}=p-k\left( n\right) \leq 0$ if $p\in (0,\sqrt{5}/5]$ and $%
t_{n}=p-k\left( n\right) \geq 0$ if $p\in \left[ 1/2,1\right] $. Thus, if $%
p\in (0,\sqrt{5}/5]$ then $f_{p}^{\prime }\left( x\right) <0$, that is, $%
f_{p}$ is decreasing, and it is derived that for $x\in \left( 0,c\right) $
with $c\in \left( 0,\pi /2\right) $%
\begin{equation*}
\ln \gamma _{p}\left( c\right) =f_{p}\left( a\right) <f_{p}\left( x\right)
<\lim_{x\rightarrow 0^{+}}f_{p}\left( x\right) =0,
\end{equation*}%
which yields (\ref{Mrre}).

Likewise, if $p\in \left[ 1/2,1\right] $ then $f_{p}^{\prime }\left(
x\right) >0$, then, $f_{p}$ is increasing, and (\ref{Mrre}) is reversed,
which completes the proof.
\end{proof}

Letting $c\rightarrow \pi /2$ and putting $p=\sqrt{5}/5$, $\sqrt{6}/6$, $1/3$%
, $0^{+}$\ in Theorem \ref{Theorem Mrre}, we get

\begin{corollary}
\label{Corollary Mrre1}$\allowbreak $The following inequalities 
\begin{eqnarray}
\gamma _{1/\sqrt{5}}\left( \tfrac{\pi }{2}\right) \left( \cos \tfrac{x}{%
\sqrt{5}}\right) ^{5/3} &<&\frac{\sin x}{x}<\left( \cos \tfrac{x}{\sqrt{5}}%
\right) ^{5/3},  \label{Mrre1} \\
\gamma _{1/\sqrt{6}}\left( \tfrac{\pi }{2}\right) \left( \cos \tfrac{x}{%
\sqrt{6}}\right) ^{2} &<&\frac{\sin x}{x}<\left( \cos \tfrac{x}{\sqrt{6}}%
\right) ^{2},  \label{Mrre2} \\
\gamma _{1/3}\left( \tfrac{\pi }{2}\right) \left( \cos \tfrac{x}{3}\right)
^{3} &<&\frac{\sin x}{x}<\left( \cos \tfrac{x}{3}\right) ^{3},  \label{Mrre3}
\\
\gamma _{0^{+}}\left( \tfrac{\pi }{2}\right) e^{-x^{2}/6} &<&\frac{\sin x}{x}%
<e^{-x^{2}/6}  \label{Mrre4}
\end{eqnarray}%
hold true for $x\in \left( 0,\pi /2\right) $, where $\gamma _{1/\sqrt{5}%
}\left( \pi /2\right) =\allowbreak 0.998\,72...$, $\gamma _{1/\sqrt{6}%
}\left( \pi /2\right) =\allowbreak 0.991\,41...$, $\gamma _{1/3}\left( \pi
/2\right) =16\sqrt{3}/\left( 9\pi \right) $, $\gamma _{0^{+}}\left( \pi
/2\right) =2e^{\pi ^{2}/24}/\pi $ are the best possible constants..
\end{corollary}

Letting $c\rightarrow \pi /2$ and putting $p=1/2\,$\ in Theorem \ref{Theorem
Mrre}, we obtain

\begin{corollary}
\label{Corollary Mlre1}$\allowbreak $For $x\in \left( 0,\pi /2\right) $, the
double inequality 
\begin{equation}
\left( \cos \tfrac{x}{2}\right) ^{4/3}<\frac{\sin x}{x}<\gamma _{1/2}\left( 
\tfrac{\pi }{2}\right) \left( \cos \tfrac{x}{2}\right) ^{4/3}  \label{Mlre1}
\end{equation}%
holds, where $1$ and $\gamma _{1/2}\left( \pi /2\right) =2^{5/3}/\pi
=\allowbreak 1.\,\allowbreak 010\,6...$ are the best constants.
\end{corollary}

\begin{remark}
Note that the first inequality of (\ref{Mlre1}) also holds for $x\in \left(
0,\pi \right) $. Indeed, differentiation yields 
\begin{equation*}
\frac{x\sin x}{\cos x+2}f_{1/2}^{\prime }\left( x\right) =\frac{\allowbreak x%
}{3}-\frac{\sin x}{\cos x+2}=g^{\prime }\left( x\right) .
\end{equation*}%
From the proof of Theorem \ref{Theorem Mr2} we see that for $x\in \left(
0,\infty \right) $, $g^{\prime }\left( x\right) >0$, which yields for $x\in
\left( 0,\pi \right) $, $f_{1/2}^{\prime }\left( x\right) >0$, and then $%
f_{1/2}\left( x\right) >f_{1/2}\left( 0^{+}\right) =0$, that is, the first
inequality of (\ref{Mlre1}) holds for $x\in \left( 0,\pi \right) $.
\end{remark}

\section{Applications}

As simple applications of main results, we will present some precise
estimates for certain integrals in this section. The following is a direct
corollary of Theorem \ref{Theorem Mrre}.

\begin{application}
We have 
\begin{equation}
\frac{1}{p}\left( \frac{2}{\pi }\right) ^{3p^{2}}\tan \frac{p\pi }{2}%
<\int_{0}^{\pi /2}\left( \frac{\sin x}{x}\right) ^{3p^{2}}dx<\int_{0}^{\pi
/2}\left( \cos px\right) =\frac{1}{p}\sin \frac{p\pi }{2}  \label{A1}
\end{equation}%
if $p\in (0,\sqrt{5}/5]$. Inequalities (\ref{A1}) is reversed if $p\in \left[
1/2,1\right] $.
\end{application}

By integrating both sides of (\ref{Mr2re}) over $\left[ 0,a\right] $ and
simple computation, we have

\begin{application}
For $a>0$ the following inequalities%
\begin{equation}
\frac{2a+\sin a}{\left( 2+\cos a\right) e^{a^{2}/6}}%
<\int_{0}^{a}e^{-x^{2}/6}dx<\allowbreak \frac{2a+\sin a}{3}  \label{A2}
\end{equation}%
are valid. Particularly, we have 
\begin{eqnarray*}
\frac{\pi +1}{2e^{\pi ^{2}/24}} &<&\int_{0}^{\pi /2}e^{-x^{2}/6}dx<\frac{\pi
+1}{3}, \\
\allowbreak \frac{\left( 4-\sqrt{2}\right) \left( \pi +\sqrt{2}\right) }{%
14e^{\pi ^{2}/96}} &<&\int_{0}^{\pi /4}e^{-x^{2}/6}dx<\frac{\pi +\sqrt{2}}{6}%
.
\end{eqnarray*}
\end{application}

For the estimate for the sine integral defined by 
\begin{equation*}
\func{Si}\left( x\right) =\int_{0}^{x}\frac{\sin t}{t}dt,
\end{equation*}%
there has some results, for example, Qi \cite{Qi.12(4)(1996)} showed that 
\begin{equation*}
\allowbreak 1.\,\allowbreak 333\,3...=\frac{4}{3}<\func{Si}\left( \frac{\pi 
}{2}\right) <\frac{\pi +1}{3}=\allowbreak 1.\,\allowbreak 380\,5...\text{;}
\end{equation*}%
the following two estimations are due to Wu \cite{Wu.19(12)(2006)}, \cite%
{Wu.12(2)(2008)}: 
\begin{eqnarray*}
\allowbreak 1.\,\allowbreak 356\,9... &=&\frac{\pi +5}{6}<\func{Si}\left( 
\frac{\pi }{2}\right) <\frac{\pi +1}{3}=\allowbreak 1.\,\allowbreak 380\,5...%
\text{,} \\
1.\,\allowbreak 368\,8... &=&\frac{92-\pi ^{2}}{60}<\func{Si}\left( \frac{%
\pi }{2}\right) <\frac{8+4\pi }{15}=\allowbreak 1.\,\allowbreak 371\,1...%
\text{.}
\end{eqnarray*}%
Now we give a more better one.$\allowbreak $

\begin{application}
We have 
\begin{equation}
\allowbreak \frac{\sqrt{3}}{4}\pi <\int_{0}^{\pi /2}\frac{\sin x}{x}dx<\frac{%
7}{16}\pi  \label{A3}
\end{equation}
\end{application}

\begin{proof}
By Corollary \ref{Corollary M} we see that the inequalities%
\begin{equation}
\cos \frac{x}{\sqrt{3}}<\frac{\sin x}{x}<\cos ^{2}\frac{x}{\sqrt{6}}
\label{4.1}
\end{equation}%
hold for $x\in \left[ 0,\pi /2\right] $. Integrating both sides over $\left[
0,\pi /2\right] $ and simple calculation yield%
\begin{equation}
\sqrt{3}\sin \frac{\pi }{2\sqrt{3}}<\int_{0}^{\pi /2}\frac{\sin x}{x}dx<%
\frac{\pi }{4}+\tfrac{\sqrt{6}}{4}\sin \frac{\pi }{\sqrt{6}}.  \label{4.2}
\end{equation}%
Using (\ref{4.1}) again gives 
\begin{eqnarray*}
\sin \frac{\pi }{2\sqrt{3}} &>&\frac{\pi }{2\sqrt{3}}\cos \tfrac{\tfrac{\pi 
}{2\sqrt{3}}}{\sqrt{3}}=\frac{\pi }{4}, \\
\sin \frac{\pi }{\sqrt{6}} &<&\frac{\pi }{\sqrt{6}}\cos ^{2}\tfrac{\tfrac{%
\pi }{\sqrt{6}}}{\sqrt{6}}=\frac{3\pi }{4\sqrt{6}},
\end{eqnarray*}%
which implies that the left hand side of (\ref{4.2}) is grater than $\sqrt{3}%
\pi /4$ and the right hand side is less than 
\begin{equation*}
\frac{\pi }{4}+\frac{\sqrt{6}}{4}\frac{3\pi }{4\sqrt{6}}=\allowbreak \frac{7%
}{16}\pi .
\end{equation*}%
Thus (\ref{A2}) follows.
\end{proof}

It is known that 
\begin{equation*}
\int_{0}^{\pi /2}\ln \left( \sin x\right) dx=\allowbreak -\frac{\pi }{2}\ln
2.
\end{equation*}%
We now evaluate the integral $\int_{0}^{c}\ln \left( \sin x\right) dx$ $%
(c\in \left( 0,\pi /2\right) $.

\begin{application}
For $c\in \left( 0,\pi /2\right) $, we have%
\begin{equation}
\allowbreak c\ln \left( \sin c\right) -c+\frac{1}{9}c^{3}<\int_{0}^{c}\ln
\left( \sin x\right) dx<c\ln c-c-\frac{1}{18}c^{3}.  \label{A4}
\end{equation}%
Particularly, we get 
\begin{equation}
\allowbreak -\frac{\pi }{72}\left( 36-\pi ^{2}\right) <\int_{0}^{\pi /2}\ln
\left( \sin x\right) dx<\allowbreak -\frac{\pi }{2}\left( \ln \frac{2}{\pi }+%
\frac{\pi ^{2}}{72}+1\right) ,  \label{A41}
\end{equation}%
\begin{equation}
-\allowbreak \tfrac{\pi }{8}\left( 2+\ln 2-\tfrac{\pi ^{2}}{72}\right)
<\int_{0}^{\pi /4}\ln \left( \sin x\right) dx<-\tfrac{\pi }{4}\left( 2\ln
2+1+\tfrac{\pi ^{2}}{288}-\ln \pi \right) .  \label{A42a}
\end{equation}
\end{application}

\begin{proof}
Letting $p\rightarrow 0^{+}$ in (\ref{Mrre}) gives 
\begin{equation*}
\gamma _{0^{+}}\left( c\right) e^{-x^{2}/6}<\frac{\sin x}{x}<e^{-x^{2}/6},
\end{equation*}%
where $\gamma _{0^{+}}\left( c\right) =c^{-1}\left( \sin c\right)
e^{c^{2}/6} $. Multiplying both sides by $x$ and taking the logarithm and
next integrating $\left[ 0,c\right] $ yield 
\begin{equation*}
\int_{0}^{c}\ln \left( \gamma _{0^{+}}\left( c\right) xe^{-x^{2}/6}\right)
dx<\int_{0}^{c}\ln \left( \sin x\right) dx<\int_{0}^{c}\ln \left(
xe^{-x^{2}/6}\right) dx.
\end{equation*}%
Simple integral computation leads to desired result.
\end{proof}

The Catalan constant \cite{Catalan.31(7)}%
\begin{equation*}
G=\sum_{n=0}^{\infty }\frac{\left( -1\right) ^{n}}{\left( 2n+1\right) ^{2}}%
=0.9159655941772190...
\end{equation*}%
is a famous mysterious constant appearing in many places in mathematics and
physics. Its integral representations \cite{Bradley.2001} contain the
following%
\begin{eqnarray*}
G &=&\int_{0}^{1}\frac{\arctan x}{x}dx=\frac{1}{2}\int_{0}^{\pi /2}\frac{x}{%
\sin x}dx \\
&=&-2\int_{0}^{\pi /4}\ln \left( 2\sin x\right) dx=\frac{\pi ^{2}}{16}-\frac{%
\pi }{4}\ln 2+\int_{0}^{\pi /4}\frac{x^{2}}{\sin ^{2}x}dx.
\end{eqnarray*}%
We next prove three accurate estimations for $G$.

\begin{application}
We have%
\begin{eqnarray}
\frac{\sqrt{6}\pi }{2\sqrt{16\sqrt{3}-\pi ^{2}}} &<&G<\allowbreak \frac{3}{32%
}\pi ^{2},  \label{A5} \\
\frac{\pi }{2}\left( \ln 2-\ln \pi +\frac{\pi ^{2}}{288}+1\right) &<&G<\frac{%
\pi }{4}\left( 2-\ln 2-\frac{\pi ^{2}}{72}\right) ,  \label{A6} \\
\frac{\pi ^{2}}{16}-\frac{\pi }{4}\ln 2+\frac{8}{5}\left( 172-99\sqrt{3}%
\right) &<&G<\allowbreak \frac{\pi ^{2}}{320}\left( 37+6\sqrt{3}\right) -%
\frac{\pi }{4}\ln 2.  \label{A7}
\end{eqnarray}
\end{application}

\begin{proof}
(i) (\ref{Mrre2}) implies that for $\left( 0,\pi /2\right) $ 
\begin{equation*}
\frac{1}{\cos ^{2}\tfrac{x}{\sqrt{6}}}<\frac{x}{\sin x}<\frac{1}{\gamma _{1/%
\sqrt{6}}\left( \tfrac{\pi }{2}\right) \cos ^{2}\tfrac{x}{\sqrt{6}}},
\end{equation*}%
where $\gamma _{1/\sqrt{6}}\left( \pi /2\right) =2\pi ^{-1}\cos ^{-2}\left( 
\sqrt{6}\pi /12\right) $. Integrating both sides over $\left[ 0,\pi /2\right]
$ yields%
\begin{equation*}
\sqrt{6}\tan \frac{\pi }{2\sqrt{6}}<\int_{0}^{\pi /2}\frac{x}{\sin x}dx<%
\frac{\sqrt{6}}{4}\pi \sin \frac{\pi }{\sqrt{6}}.
\end{equation*}%
By Corollary \ref{Corollary M} it is seen that for $x\in \left( 0,\pi
/2\right) $ 
\begin{equation*}
\sin x>x\left( \cos \frac{\sqrt{6}x}{3}\right) ^{1/2}
\end{equation*}%
holds, and so 
\begin{equation*}
\tan ^{2}\tfrac{\pi }{2\sqrt{6}}=\tfrac{1}{1-\sin ^{2}\tfrac{\pi }{2\sqrt{6}}%
}-1>\tfrac{1}{1-\left( \tfrac{\pi }{2\sqrt{6}}\right) ^{2}\cos \left( \tfrac{%
\sqrt{6}}{3}\tfrac{\pi }{2\sqrt{6}}\right) }-1=\allowbreak \tfrac{\pi ^{2}}{%
16\sqrt{3}-\pi ^{2}}.
\end{equation*}%
On the other hand, application of the second inequality in (\ref{Mrre2})
gives 
\begin{equation*}
\sin \frac{\pi }{\sqrt{6}}<\frac{\pi }{\sqrt{6}}\cos ^{2}\tfrac{\tfrac{\pi }{%
\sqrt{6}}}{\sqrt{6}}=\allowbreak \frac{\sqrt{6}}{8}\pi .
\end{equation*}%
Hence, 
\begin{equation*}
\tfrac{\sqrt{6}\pi }{\sqrt{16\sqrt{3}-\pi ^{2}}}<\sqrt{6}\tan \tfrac{\pi }{2%
\sqrt{6}}<\int_{0}^{\pi /2}\frac{x}{\sin x}dx<\frac{\sqrt{6}}{4}\pi \sin 
\tfrac{\pi }{\sqrt{6}}<\frac{3}{16}\pi ^{2},
\end{equation*}%
which, from the second integral representation for $G$, implies (\ref{A5}).

(ii) By (\ref{A42}) it is derived that 
\begin{equation*}
\frac{\pi }{8}\left( \ln 2+\frac{1}{72}\pi ^{2}-2\right) <\int_{0}^{\pi
/4}\ln \left( 2\sin x\right) dx<\frac{\pi }{4}\left( \ln \pi -\ln 2-1-\frac{1%
}{288}\pi ^{2}\right) ,
\end{equation*}%
it follows from the third integral representation for $G$ that (\ref{A5})
holds.

(iii) Lastly, we use the fourth integral representation for $G$ to prove (%
\ref{A6}). Employing Theorem \ref{Theorem Mrre}, we have 
\begin{equation*}
\gamma _{1/3}\left( \tfrac{\pi }{4}\right) \cos ^{3}\frac{x}{3}<\frac{\sin x%
}{x}<\cos ^{3}\frac{x}{3},
\end{equation*}%
where $\gamma _{1/3}\left( \pi /4\right) =16\left( 3\sqrt{3}-5\right) /\pi $%
. It is obtained that 
\begin{equation*}
\cos ^{-6}\frac{x}{3}<\frac{x^{2}}{\sin ^{2}x}<\allowbreak \frac{\pi ^{2}}{%
512}\left( 15\sqrt{3}+26\right) \cos ^{-6}\frac{x}{3},
\end{equation*}%
and integrating both sides over $\left[ 0,\pi /4\right] $ leads to 
\begin{equation*}
\int_{0}^{\pi /4}\cos ^{-6}\frac{x}{3}dx<\int_{0}^{\pi /4}\frac{x^{2}}{\sin
^{2}x}dx<\allowbreak \frac{\pi ^{2}}{512}\left( 15\sqrt{3}+26\right)
\int_{0}^{\pi /4}\cos ^{-6}\frac{x}{3}dx.
\end{equation*}%
Integral computation reveals that 
\begin{equation*}
\int_{0}^{\pi /4}\left( \cos \tfrac{x}{3}\right) ^{-6}dx=\frac{8}{5}\left(
172-99\sqrt{3}\right) ,\allowbreak
\end{equation*}%
and therefore 
\begin{equation*}
\frac{8}{5}\left( 172-99\sqrt{3}\right) <\int_{0}^{\pi /4}\frac{x^{2}}{\sin
^{2}x}dx<\allowbreak \frac{\pi ^{2}}{320}\allowbreak \left( 6\sqrt{3}%
+17\right) .
\end{equation*}%
Application of the fourth integral representation for $G$ the desired
inequality (\ref{A6}) follows.
\end{proof}

$\allowbreak $

\end{document}